\documentclass[sn-mathphys,Numbered]{sn-jnl}% Math and Physical Sciences Reference Style
\usepackage{graphicx}%
\usepackage{multirow}%
\usepackage{amsmath,amssymb,amsfonts}%
\usepackage{amsthm}%
\usepackage{mathrsfs}%
\usepackage[title]{appendix}%
\usepackage{xcolor}%
\usepackage{textcomp}%
\usepackage{manyfoot}%
\usepackage{booktabs}%
\usepackage{algorithm}%
\usepackage{algorithmicx}%
\usepackage{algpseudocode}%
\usepackage{listings}%
\theoremstyle{thmstyleone}%
\newtheorem{theorem}{Theorem}%  meant for continuous numbers
\newtheorem{corollary}{Corollary}%  meant for continuous numbers
\newtheorem{lemma}{Lemma}%  meant for continuous numbers
\newtheorem{proposition}{Proposition}% 

\theoremstyle{thmstyletwo}%
\newtheorem{example}{Example}%
\newtheorem{remark}{Remark}%

\theoremstyle{thmstylethree}%
\newtheorem{definition}{Definition}%

\begin{document}
	
	\title[Results on continuous $K$-frames for quaternionic (super) Hilbert spaces]{Results on Continuous $K$-frames for Quaternionic (Super) Hilbert Spaces}
	
	
	\author{\fnm{Najib} \sur{Khachiaa}}\email{khachiaa.najib@uit.ac.ma}

	\affil{\orgdiv{Laboratory Partial Differential Equations, Spectral Algebra and Geometry, Department of Mathematics}, \orgname{Faculty of Sciences, University Ibn Tofail}, \orgaddress{\city{Kenitra}, \country{Morocco}}}

	\abstract{This paper aims to explore the concept of continuous \( K \)-frames in quaternionic Hilbert spaces. First, we investigate \( K \)-frames in a single quaternionic Hilbert space \( \mathcal{H} \), where \( K \) is a right $\mathbb{H}$-linear  bounded  operator acting on \( \mathcal{H} \). Then, we extend the research to  two quaternionic Hilbert spaces, \( \mathcal{H}_1 \) and \( \mathcal{H}_2 \), and study \( K_1 \oplus K_2 \)-frames for the super quaternionic Hilbert space \( \mathcal{H}_1 \oplus \mathcal{H}_2 \), where \( K_1 \) and \( K_2 \) are right $\mathbb{H}$-linear bounded operators on \( \mathcal{H}_1 \) and \( \mathcal{H}_2 \), respectively. We examine the relationship between the continuous \( K_1 \oplus K_2 \)-frames and the continuous \( K_1 \)-frames for \( \mathcal{H}_1 \) and the continuous \( K_2 \)-frames for \( \mathcal{H}_2 \). Additionally, we explore the duality between the continuous \( K_1 \oplus K_2 \)-frames and the continuous \( K_1 \)- and \( K_2 \)-frames individually.}

	\keywords{Continuous $K$-frames, Super continuous  $K$-frames, Quaternionic Hilbert spaces, Quaternionic super Hilbert spaces}

	\pacs[MSC Classification]{42C15; 42C40; 47A05.}
	
	\maketitle

\section{Introduction and preliminaries}
Frames in quaternionic Hilbert spaces provide a robust framework for analyzing and reconstructing signals within a higher-dimensional space. As an extension of the classical frame theory, these structures facilitate efficient data representation and processing. The unique properties of quaternionic spaces offer new avenues for exploration, particularly in applications such as signal processing and communications. Recently, the theory of frames has undergone several generalizations, including g-frames, $K$-frames and continuous frames. These advancements have broadened the applicability of frame theory in various mathematical contexts. In this work, we will study the theory of continuous $K$-frames for quaternionic Hilbert spaces and  will also study a particular case, which is the theory of continuous $K$-frames for  quaternionic super Hilbert spaces. The quaternionic field is an extension of the real and complex number systems, consisting of numbers known as quaternions. Quaternions are used to represent three-dimensional rotations and orientations, making them invaluable in computer graphics and robotics. Unlike real and complex numbers, quaternion multiplication is non-commutative, which adds complexity to their algebraic structure. Quaternions also provide a more efficient way to perform calculations in three-dimensional space, enhancing applications in physics and engineering.

\begin{definition}[The field of quaternions]\cite{11}
The non-commutative field of quaternions $\mathbb{H}$ is a four-dimensional real algebra with unity. In $\mathbb{H}$, $0$ denotes the null element  and $1$ denotes the identity with respect to multiplication. It also includes three so-called imaginary units, denoted by $i,j,k$. i.e.,
$$\mathbb{H}=\{a_0+a_1i+a_2j+a_3k:\; a_0,a_1,a_2,a_3\in \mathbb{R}\},$$
where $i^2=j^2=k^2=-1$, $ij=-ji=k$, $jk=-kj=i$ and $ki=-ik=j$. For each quaternion $q=a_0+a_1i+a_2j+a_3k$, we deifine the conjugate of $q$ denoted by $\overline{q}=a_0-a_1i-a_2j-a_2k \in \mathbb{H}$ and the module of $q$ denoted by $\vert q\vert $ as 
$$\vert q\vert =(\overline{q}q)^{\frac{1}{2}}=(q\overline{q})^{\frac{1}{2}}=\displaystyle{\sqrt{a_0^2+a_1^2+a_2^2+a_3^2}}.$$
For every $q\in \mathbb{H}$, $q^{-1}=\displaystyle{\frac{\overline{q}}{\vert q\vert ^2}}.$
\end{definition}

\begin{definition}[ Right quaternionic vector space]\cite{11}
A right quaternioniq vector space $V$ is a linear vector space under right scalar multiplication over the field of quaternions $\mathbb{H}$, i.e., the right scalar multiplication 
$$\begin{array}{rcl}
V\times \mathbb{H} &\rightarrow& V\\
(v,q)&\mapsto& v.q,
\end{array}$$
satisfies the following for all $u,v\in V$ and $q,p\in \mathbb{H}$:
\begin{enumerate}
\item $(v+u).q=v.q+u.q$,
\item $v.(p+q)=v.p+v.q$,
\item $v.(pq)=(v.p).q$.
\end{enumerate}
Instead of $v.q$, we often use the notation $vq$.
\end{definition}

\begin{definition}[Right quaterninoic pre-Hilbert space]\cite{11}
A right quaternionic pre-Hilbert space $\mathcal{H}$, is a right quaternionic vector space equipped with the binary mapping\\ \( \langle \cdot \mid \cdot \rangle : \mathcal{H} \times \mathcal{H} \to \mathbb{H} \) (called the Hermitian quaternionic inner product) which satisfies the following properties:

\begin{itemize}
    \item[(a)] \( \langle v_1 \mid v_2 \rangle = \overline{\langle v_2 \mid v_1 \rangle} \) for all \( v_1, v_2 \in \mathcal{H}\),
    \item[(b)] \( \langle v \mid v \rangle > 0 \) if \( v \neq 0 \),
    \item[(c)] \( \langle v \mid v_1 + v_2 \rangle = \langle v \mid v_1 \rangle + \langle v \mid v_2 \rangle \) for all \( v, v_1, v_2 \in \mathcal{H} \),
    \item[(d)] \( \langle v \mid uq \rangle = \langle v \mid u \rangle q \) for all \( v, u \in \mathcal{H} \) and \( q \in \mathbb{H} \).
\end{itemize}
\end{definition}

In view of Definition  $3$ , a right pre-Hilbert space $\mathcal{H}$ also has the property:

\begin{itemize}
    \item[(i)] $ \langle vq \mid u \rangle = \overline{q} \langle v \mid u \rangle$ for all $v, u \in \mathcal{H} $ and \( q \in \mathbb{H} \).
\end{itemize}

Let \( \mathcal{H} \) be a right quaternionic pre-Hilbert space with the Hermitian inner product \( \langle \cdot \mid \cdot \rangle \). Define the quaternionic norm \( \|\cdot\| : \mathcal{H} \to \mathbb{R}^+ \) on \( \mathcal{H} \) by
\[
\|u\| = \sqrt{\langle u \mid u \rangle}, \quad u \in \mathcal{H}, 
\]
which satisfies the following properties:
\begin{enumerate}
\item $\|uq\|=\|u\|\vert q\vert$, for all $u\in \mathcal{H}$ and $q\in \mathbb{H}$,
\item $\| u+v\|\leq \|u\|+\|v\|$,
\item $\|u\|=0\Longleftrightarrow u=0$ for $u\in \mathcal{H}$.
\end{enumerate}
\begin{definition}[ Right quaternionic Hilbert space]\cite{11}
A right quaternionic pre-Hilbert space is called a right quaternionic Hilbert space if it is complete with respect to the quaternionic norm.
\end{definition}

\begin{example}
Define $$L^2(\Omega,\mathbb{H}):=\left\{ g:\Omega\rightarrow \mathcal{H} \text{ measurable }:\; \displaystyle{\int_{\Omega}\vert g(\omega) \vert^2 d\mu(\omega)}<\infty\right\}.$$
$L^2(\Omega,\mathbb{H})$ under right multiplication by quaternionic scalars together with the quaternionic inner product defined as: $\langle g \mid h\rangle:=\displaystyle{\int_{\Omega} \overline{g(\omega)}h(\omega)}d\mu(\omega)$ for  $g,h\in L^2(\Omega,\mathbb{H})$, is a right quaternionic Hilbert space.
\end{example}
\begin{theorem}[The Cauchy-Schwarz inequality]\cite{11}
If $\mathcal{H}$ is a right quaternionic Hilbert space, then for all $u,v\in \mathcal{H}$, 
$$\vert \langle u\mid v\rangle \vert\leq \|u\| \|v\|.$$
\end{theorem}
\begin{definition}[orthogonality]\cite{11}
Let \(\mathcal{H} \) be a right quaternionic Hilbert space and \( A\) be a subset of \( \mathcal{H} \). Then, define the set:

\begin{itemize}
    \item \( A^{\perp} = \{ v \in \mathcal{H} : \langle v \mid u \rangle = 0 \; \forall \; u \in A \} \);
    \item \( \langle A \rangle \) as the right quaternionic vector subspace of \( \mathcal{H} \) consisting of all finite right \( \mathbb{H} \)-linear combinations of elements of \( A\).
\end{itemize}
\end{definition}

\begin{proposition}\cite{11}
Let \(\mathcal{H} \) be a right quaternionic Hilbert space and \( A\) be a subset of \( \mathcal{H} \). Then,
\begin{enumerate}
\item $A^{\perp}=\langle A\rangle^\perp=\overline{\langle A\rangle }^\perp=\overline{\langle A\rangle ^\perp}.$
\item $(A^\perp)^\perp=\overline{\langle A\rangle}.$
\item $\overline{A}\oplus A^\perp=\mathcal{H}.$
\end{enumerate}
\end{proposition}

\begin{theorem}\cite{11}
Let \( \mathcal{H} \) be a quaternionic Hilbert space and let \( N \) be a subset of \( \mathcal{H} \) such that, for \( z, z' \in N \), we have \( \langle z \mid z' \rangle = 0 \) if \( z \neq z' \) and \( \langle z \mid z \rangle = 1 \). Then, the following conditions are equivalent:

\begin{itemize}
    \item[(a)] For every \( u, v \in \mathcal{H} \), the series \( \sum_{z \in N} \langle u \mid z \rangle \langle z \mid v \rangle \) converges absolutely and
    \[
    \langle u \mid v \rangle = \sum_{z \in N} \langle u \mid z \rangle \langle z \mid v \rangle;
    \]
    
    \item[(b)] For every \( u \in \mathcal{H} \), \( \|u\|^2 = \displaystyle{\sum_{z \in N} |\langle z \mid u \rangle|^2 }\);
    
    \item[(c)] \( N^{\perp} = \{0\} \);
    
    \item[(d)] \( \langle N \rangle \) is dense in \( \mathcal{H} \).
\end{itemize}
\end{theorem}
\begin{definition}\cite{11}
A subset $N$ of $\mathcal{H}$ that satisfies one of the statements in Theorem $2$ is called Hilbert basis or orthonormal basis for $\mathcal{H}$.
\end{definition}
\begin{theorem}\cite{11}
Every quaternionic Hilbert space has a Hilbert basis.\\
\end{theorem}
\begin{definition}[Right $\mathbb{H}$-linear operator]\cite{11}
Let $\mathcal{H}$ and $\mathcal{K}$ be two right quaternionic Hilbert spaces. Let  $L:\mathcal{H}\rightarrow \mathcal{K}$ be a map.
\begin{enumerate}
\item  $L$ is said to be right $\mathbb{H}$-linear operator if $L(uq+vp)=L(u)q+L(v)p$ for all $u,v\in \mathcal{H}$ and $p,q\in \mathcal{H}$.
\item  If $L$ is  a right $\mathbb{H}$-linear operator. $L$ is continuous if and only if $L$ is bounded; i.e., there exists $M> 0$ such that for all $u\in \mathcal{H}$, 
$$\| Lu\|\leq M\| u\|.$$
We denote $\mathbb{B}(\mathcal{H},\mathcal{K})$ the set of all right $\mathbb{H}$-linear bounded operators from $\mathcal{H}$ to $\mathcal{K}$, and if $\mathcal{H}=\mathcal{K}$, we denote $\mathbb{B}(\mathcal{H})$ instead of $\mathbb{B}(\mathcal{H},\mathcal{H})$.
\item If $L$ is a right $\mathbb{H}$-linear  bounded operator, we define the norm of $L$ as: 
$$\|L\|=\displaystyle{\sup_{\|u\|=1}\|Lu\|}=\inf\{M>0:\; \| Lu\|\leq M\| u\|,\; \forall u\in \mathcal{H}\}.$$
And we have for all $L,M\in \mathbb{B}(\mathcal{H}), \|L+M\|\leq\|L\|+\|M\|$ and $\|MN\|\leq\|L\|\|M\|.$\\
\end{enumerate}
\end{definition}
\begin{definition}[Adjoint operator]\cite{11}
Let $\mathcal{H}$ be a right quaternionic Hilbert space and  $L\in \mathbb{B}(\mathcal{H})$. The adjoint operator of $L$, denoted $L^*$, is the unique operator in $\mathbb{B}(\mathcal{H})$ satisfying for all $u,v\in \mathcal{H}$:
$$\langle Lu\mid v\rangle=\langle u\mid L^*v\rangle.$$
\end{definition}
\begin{theorem}[ The closed graph theorem]\label{thm7} \cite{11}
Let $\mathcal{H}, \mathcal{K}$ be two right quaternionic Hilbert spaces and let $L:\mathcal{H}\rightarrow \mathcal{K}$ be a right $\mathbb{H}$-linear opeartor. If $Graph(L)$ is closed, then $L\in \mathbb{B}(\mathcal{H},\mathcal{K})$.\\
\end{theorem}
\begin{theorem}[Quaternionic representation Riesz' theorem]\label{Riesz}\cite{10}
    If $\mathcal{H}$ is a right  quaternionic Hilbert space, the map
    \[
          v\in \mathcal{H} \mapsto \langle v \mid \cdot \rangle \in \mathcal{H}'
    \]
    is well-posed and defines a conjugate-$\mathbb{H}$-linear isomorphism.\\
\end{theorem}
\begin{definition}[Continuous frames]
Let $(\Omega,\mu)$ be a measure space, $F:\Omega\rightarrow \mathcal{H}$ be a weakly measurable mapping. $F$ is said to be continuous frame for $\mathcal{H}$ if there exist $0<A\leq B<\infty$ such that for all $u\in \mathcal{H}$, the following inequality holds:
$$A\|u\|^2\leq \displaystyle{\int_{\Omega}\vert \langle F(\omega),u\rangle \vert^2 d\mu(\omega)}\leq B\|u\|^2.$$
\end{definition}

\begin{definition}[Continuous $K$-Frames]
Let $(\Omega,\mu)$ be a measure space, $F:\Omega\rightarrow \mathcal{H}$ be a weakly measurable mapping, $K\in \mathbb{B}(\mathcal{H})$. $F$ is said to be continuous $K$-frame for $\mathcal{H}$ if there exist $0<A\leq B<\infty$ such that for all $u\in \mathcal{H}$, the following inequality holds:
$$A\|K^*u\|^2\leq \displaystyle{\int_{\Omega}\vert \langle F(\omega),u\rangle \vert^2 d\mu(\omega)}\leq B\|u\|^2.$$
\begin{enumerate}
\item If only the upper inequality holds, $F$ is called a Bessel mapping for $\mathcal{H}$.
\item If $A=B=1$, $F$ is called a Parseval continuous $K$-frame for $\mathcal{H}$.
\end{enumerate}
\end{definition}
Thanks to the definition of a Bessel mapping, the following definition is well justified.
\begin{definition}[The transform operator]
Let $(\Omega, \mu)$ be a measure space and  $F:\Omega\rightarrow \mathcal{H}$ be a Bessel mapping for $\mathcal{H}$. The transform operator of $F$ is the right $\mathbb{H}$-linear bounded operator denoted by $\theta$ and  defined as follows: $$\begin{array}{rcl}
\theta: \mathcal{H}&\rightarrow& L^2(\Omega,\mathbb{H})\\
u&\mapsto& \theta (u),
\end{array}$$
where $\theta(u)\omega:= \langle F\omega ,u\rangle$ for all $\omega \in \Omega$.\\
\end{definition}

Let $u\in \mathcal{H}$ and $g\in L^2(\Omega,\mathbb{H})$, we have: 
$$\begin{array}{rcl}
\langle \theta(u), g\rangle&=&\displaystyle{\int_{\Omega} \overline{\theta(u)}gd\mu}\\
&=&\displaystyle{\int_{\Omega} \overline{\langle F\omega,u\rangle} g(\omega)d\mu(\omega)}\\
&=&\displaystyle{\int_{\Omega} \langle u,F\omega\rangle g(\omega)d\mu(\omega)}\\
&=&\displaystyle{\int_{\Omega} \langle u, F(\omega)g(\omega)\rangle d\mu(\omega)}.\\
\end{array}$$
By the quaternionic represention Riesz' theorem \ref{Riesz}, there exists a unique vector in $\mathcal{H}$, denoted by $\displaystyle{\int_{\Omega}F(\omega)g(\omega)d\mu(\omega)}$, such that for all $u\in \mathcal{H}$, we have: \begin{equation}
\displaystyle{\int_{\Omega} \langle u, F(\omega)g(\omega)\rangle d\mu(\omega)}=\left\langle u, \displaystyle{\int_{\Omega}F(\omega)g(\omega)d\mu(\omega)}\right\rangle.
\end{equation}
Hence, the adjoint operator of $\theta$ is defined as follows:
$$\begin{array}{rcl}
\theta^*:\mathcal{H}&\rightarrow& L^2(\Omega,\mathbb{H})\\
u&\mapsto& \displaystyle{\int_{\Omega}F(\omega)g(\omega)d\mu(\omega)}.
\end{array}$$

\begin{definition}[The pre-frame operator]
Let $(\Omega, \mu)$ be a measure space and  $F:\Omega\rightarrow \mathcal{H}$ be a Bessel mapping for $\mathcal{H}$. The pre-frame operator of $F$, denoted by $T$, is the adjoint of its transform operator, i.e., $T=\theta^*$.
\end{definition}

\begin{definition}[The frame operator]
Let $(\Omega, \mu)$ be a measure space and  $F:\Omega\rightarrow \mathcal{H}$ be a Bessel mapping for $\mathcal{H}$. The frame operator of $F$, denoted by $S$, is the composite of its pre-frame operator and its transform operator, i.e., $S=T\theta$. By respecting the previous notation in $(1)$, $S$ is defined as follows: $$\begin{array}{rcl}
S:\mathcal{H}&\rightarrow&\mathcal{H}\\
u&\mapsto&  \displaystyle{\int_{\Omega}F(\omega)\langle F(\omega),u\rangle d\mu(\omega)}.
\end{array}$$
\end{definition}
\begin{remark}\label{rem1}
Let $(\Omega, \mu)$ be a measure space and  $F:\Omega\rightarrow \mathcal{H}$ be a Bessel mapping for $\mathcal{H}$. Then, for all $u\in \mathcal{H}$, we have: $$\displaystyle{\int_{\Omega}\vert \langle F(\omega),u\rangle \vert^2 d\mu(\omega)}=\langle Su,u\rangle=\|\theta u\|^2.$$
\end{remark}
The following proposition quickly follows from the definition of a frame and Remark \ref{rem1}.
\begin{proposition}
Let $(\Omega, \mu)$ be a measure space and  $F:\Omega\rightarrow \mathcal{H}$ be a continuous frame of \( \mathcal{H} \). Then:
\begin{enumerate}
    \item \( \theta \) is a right $\mathbb{H}$-linear  bounded injective operator with closed range.
    \item \( T \) is a right $\mathbb{H}$-linear bounded surjective operator.
    \item \( S \) is a right $\mathbb{H}$-linear bounded, positive, and invertible operator.
\end{enumerate}
\end{proposition}
\begin{remark}
Unlike to ordinary frames, the pre-frame operator of a continuous $K$-frame is not surjective, the transform opeartor is not injective with closed range and the frame operator is not invertible.
\end{remark}
For more details on Quaternionic calculus in Hilbert spaces, the reader can refer to \cite{1}, \cite{6'}, \cite{11} and \cite{15'}.
\section{Continuous $K$-frames for quaternionic Hilbert spaces}
In this section, let $(\Omega,\mu)$ be a measure space and  $\mathcal{H}$ be a right quaternionic Hilbert space. We will characterize continuous $K$-frames for $\mathcal{H}$ by their associated operators and provide some general results on continuous $K$-frames in $\mathcal{H}$. 
First, we present the quaternionic version of Douglas's theorem. The proof does not differ significantly from the complex case, but this does not preclude its presentation.
\begin{theorem}[Quaternionic version of Douglas's Theorem]\label{thm1}
Let $\mathcal{H}_1,\mathcal{H}_2$ and $\mathcal{H}$ be right quaternionic Hilbert spaces and let $L\in \mathbb{B}(\mathcal{H}_1,\mathcal{H})$ and $M\in \mathbb{B}(\mathcal{H}_2,\mathcal{H})$. Then, the following statements are equivalent:
\begin{enumerate}
\item $R(L)\subset R(M)$.
\item There exists a constant $c>0$ such that $LL^*\leq MM^*c$.
\item There exists $X\in \mathbb{B}(\mathcal{H}_1,\mathcal{H}_2)$ such that $L=MX$.
\end{enumerate}
\end{theorem}
\begin{proof}
Assume $3.$, i.e., there exists $X\in \mathbb{B}(\mathcal{H}_1,\mathcal{H}_2)$ such that $L=MX$. Then,  $LL^*=MXX^*M^*$. Thus, using the Cauchy-Schwarz inequality, for all $u\in \mathcal{H}$, we have: $$\langle LL^*u,u\rangle =\langle MXX^*M^*u,u\rangle \leq \|X\|^2\langle MM^*u,u\rangle.$$ Hence, $3.$ implies $2..$ It is clear that $3.$ implies $1.$. Assume $1.$, i.e., $R(L)\subset R(M)$. We can define a right $\mathbb{H}$-linear operator $X$ from $\mathcal{H}_1$ to $\mathcal{H}_2$ as follows: For $u\in \mathcal{H}_1$, $Lu\in R(L)\subset R(M)$, then there exists a unique $v\in \text{ker}(M)^\perp$ such that $Bv=Au$. Set $Xu=v$, hence $L=MX$. It remains to prove that $X$ is bounded. Let $\{u_n\}_{n\geqslant 1}\subset \mathcal{H}_1$ such that $u_n \to u$ and $Xu_n \to v$, then $Lu_n\to Lu$ and $MXu_n\to Mv$, and since $L=MX$, then $Lu_n\to Mv$. Hence, $Mv=Lu$, and since $\text{ker}(M)^\perp$ is closed, then $v\in \text{ker}(M)^\perp$, thus $Xu=v$. Hence, by the closed graph theorem, $X$ is bounded. Hence, $1.$ implies $3.$. Assume $2.$, i.e., there exists $c> 0$, such that $LL^*\leq MM^*c$. Define $D:R(M^*)\rightarrow R(L^*)$ by $D(M^*u)=L^*u$. Then $D$ is well defined and bounded since: 
\begin{equation}
\|DM^*u\|^2=\|L^*u\|^2=\langle LL^*u,u\rangle \leq c\langle MM^*u,u\rangle=c\| M^*u\|^2. 
\end{equation}
Hence, $D$ can be uniquely extended to $\overline{R(M^*)}$, and if we define $D$ on $R(M^*)^\perp$ to be zero, then $DM^*=L^*$, Hence, $L=MD^*$. Hence, $2.$ implies $3.$.
\end{proof}

The following theorem characterizes continuous $K$-frames for a right quaternionic Hilbert spaces using the associated operators.
\begin{theorem}\label{thm2}
Let $F:\Omega\rightarrow \mathcal{H}$ be a Bessel mapping and $K\in \mathbb{B}(\mathcal{H})$. Then the following statements are equivalent.
\begin{enumerate}
\item $F$ is a continuous $K$-frame for $\mathcal{H}$.
\item $R(K)\subset R(T).$
\item There exists a constant $c>0$ such that $KK^*c\leq S$.
\item There exists $X\in \mathbb{B}(\mathcal{H},L^2(\Omega,\mathbb{H})\,)$ such that $K=TX$. In this case,  such $X\in \mathbb{B}(\mathcal{H},L^2(\Omega,\mathbb{H}))$ is called a $K$-dual mapping to $F$.
\end{enumerate}
\end{theorem}
\begin{proof}
Since $F$ is a Bessel mapping, then $F$ is a continuous $K$-frame for $\mathcal{H}$, if and only if, there exist $c> 0$ such that for all $u\in \mathcal{H}$, $c\|K^*u\|^2\leq \displaystyle{\int_{\Omega}\vert \langle F(\omega),u\rangle \vert^2 d\mu(\omega)}.$ Using remark \ref{rem1},  $\{u_i\}_{i\in I}$ is a $K$-frame for $\mathcal{H}$, if and only if, there exist $c> 0$ such that  $KK^*c\leq S$, if and only if, $R(K)\subset R(T)$ by Theorem \ref{thm1}, if and only if there exists $X\in \mathbb{B}(\mathcal{H},L2(\Omega,\mathbb{H})\,)$ such that $K=TX$ (by Theorem \ref{thm1}).
\end{proof}

\begin{proposition}\label{prop2}
Let $F:\Omega\rightarrow \mathcal{H}$ be a continuous frame  and $K\in \mathbb{B}(\mathcal{H})$. Then, $KF$ is a continuous $K$-frame for $\mathcal{H}$.
\end{proposition}
\begin{proof}
Since $F$ is a continuous frame for $\mathcal{H}$, then for all $u\in \mathcal{H}$, $u=\displaystyle{\int_{\Omega}F(\omega)\langle F(\omega),S^{-1}u\rangle d\mu(\omega)}$, where $S$ is the frame operator of $F$. Let $v\in \mathcal{H}$, we have, 
$$\begin{array}{rcl}
\langle Ku,v\rangle =\langle u, K^*v\rangle &=&\left\langle \displaystyle{\int_{\Omega} F(\omega)\langle F(\omega),u\rangle d\mu(\omega)},K^*v\right\rangle\\
&=&\displaystyle{\int_{\Omega}\left\langle F(\omega)\langle F(\omega),S^{-1}u\rangle,\;K^*v\right\rangle d\mu(\omega)}\\
&=&\displaystyle{\int_{\Omega}\left\langle KF(\omega)\langle F(\omega),S^{-1}u\rangle,\;v\right\rangle d\mu(\omega)}.\\
 \end{array}$$
The, by uniqueness, $Ku=\displaystyle{\int_{\Omega}KF(\omega)\langle F(\omega),S^{-1}u\rangle d\mu(\omega)}$ ($\forall u\in \mathcal{H}$). Hence, $K=T_{KF }X$, where $T_{KF}$ is the pre-frame operator of $KF$ and  $X\in \mathbb{B}(\mathcal{H},L^2(\Omega,\mathbb{H}))$ is defined as follows: For all $u\in \mathcal{H}$ and $\omega \in \Omega$, $Xu(\omega):=\langle F(\omega),S^{-1}u\rangle$. Thus, Theorem \ref{thm2} completes the proof.
\end{proof}
\begin{proposition}
Let $K\in \mathbb{B}(\mathcal{H})$ and $F:\Omega\rightarrow \mathcal{H}$ be a continuous $K$-frame for $\mathcal{H}$. Then, for all $X\in \mathbb{B}(\mathcal{H})$, $F$ is a continuous $KX$-frame for $\mathcal{H}$. Moreover, if $G\in \mathbb{B}(\mathcal{H},L^2(\Omega,\mathbb{H}))$ is continuous $K$-dual mapping to $F$, then $GX$ is a $KX$-dual frame to $F$.
\end{proposition}
\begin{proof}
Denote by $T$ the pre-frame operator of $F$. Since $R(KX)\subset R(K)$ and $R(K)\subset R(T)$, then $R(KX) R(T)$, then $F$ is a continuous $KX$-frame for $\mathcal{H}$. Let $G\in \mathbb{B}(\mathcal{H},L^2(\Omega,\mathbb{H}))$ be a $K$-dual mapping to $G$. Then, $K=TG$. Thus, $KX=TGX$. Hence, $GX$ is a $KX$-dual mapping to $F$.
\end{proof}

\begin{definition}
Let $K\in \mathbb{B}(\mathcal{H})$ and  $F:\Omega\rightarrow \mathcal{H}$ be a continuous $K$-frame for $\mathcal{H}$. $F$ is said to be minimal if its transform operator is injective.\\
\end{definition}
\begin{remark}
If $F:\Omega\rightarrow \mathcal{H}$ is a minimal continuous $K$-frame, then $F\neq 0$ on any measurable subset of $(\Omega,\mu)$ with positive measure.\\
\end{remark}
\begin{proposition}
Let $K\in \mathbb{B}(\mathcal{H})$ and  $F:\Omega\rightarrow \mathcal{H}$ be a continuous $K$-frame for $\mathcal{H}$. If $F$ is minimal, then $F$ has a unique $K$-dual mapping.\\
\end{proposition}
\begin{proof}
Denote by $T$ the pre-frame operator of $F$ and Let $G_1,G_2\in \mathbb{B}(\mathcal{H},L^2(\Omega,\mathbb{H}))$ be two $K$-dual mappings to $F$. Then, $K=TG_1=TG_2$. Hence, $G_1=G_2$ since $T$ is injective.
\end{proof}
\section{Continuous $K$-frames on quaternionic super Hilbert spaces}
In this section, we explore the study of continuous  $K$-frames in direct sums of quaternionic Hilbert spaces, a topic that, while seemingly abstract, has significant theoretical and practical relevance. Frames in  quaternionic Hilbert spaces provide a flexible framework for representing signals or functions in a manner that allows for efficient reconstruction and analysis. By extending frame theory to direct sums of Hilbert spaces, we can address systems that are naturally composite, such as multi-channel or multi-partite systems. This approach not only generalizes results from simpler spaces to more complex, infinite-dimensional settings but also opens avenues for applications in signal processing, quantum mechanics, and data analysis. For instance, in signal processing, direct sum constructions allow for the decomposition of signals into components from different subspaces, facilitating more efficient encoding, compression, and recovery. Similarly, in quantum information theory, frames in direct sum spaces offer a tool for understanding how information is distributed and processed across subsystems. Ultimately, the study of frames in these settings contributes to a deeper understanding of the structure of Hilbert spaces and has potential applications across several fields of mathematics and physics. \\

In all this section, $(\Omega,\mu)$ is a measure space, \( \mathcal{H}_1 \) and \( \mathcal{H}_2 \) are two right quaternionic Hilbert spaces. \( \mathcal{H}_1 \oplus \mathcal{H}_2 \) is the direct sum  of \( \mathcal{H}_1 \) and \( \mathcal{H}_2 \). The space \( \mathcal{H}_1 \oplus \mathcal{H}_2 \), equipped with the inner product
\[
\langle u_1 \oplus v_1, u_2 \oplus v_2 \rangle := \langle u_1, u_2 \rangle_{\mathcal{H}_1} + \langle v_1, v_2 \rangle_{\mathcal{H}_2},
\]
for all \( u_1, u_2 \in \mathcal{H}_1 \) and \( v_1, v_2 \in \mathcal{H}_2 \), is clearly a right quaternionic Hilbert space, which we call the super right quaternionic  Hilbert space of \( \mathcal{H}_1 \) and \( \mathcal{H}_2 \). \( \mathbb{B}(\mathcal{H}_1, \mathcal{H}_2) \) denotes the collection of all right $\mathbb{H}$-linear  bounded operators from \( \mathcal{H}_1 \) to \( \mathcal{H}_2 \). For \(\mathcal{H}_1=\mathcal{H}_2=\mathcal{H}\), we denote, simply, \(\mathbb{B}(\mathcal{H})$. For \( K \in \mathbb{B}(\mathcal{H}_1) \) and \( L \in \mathbb{B}(\mathcal{H}_2) \), we denote by \( K \oplus L \) the right \(\mathbb{H}\)-linear bounded operator on \( \mathcal{H}_1 \oplus \mathcal{H}_2 \), defined for all \( u \in \mathcal{H}_1 \) and \( v\in \mathcal{H}_2 \) by:
\[
(K \oplus L)(u \oplus v) = Ku \oplus Lv.
\]

For a right $\mathbb{H}$-linear bounded operator \( L\), \( R(L) \) and \( N(L) \) denote the range and the kernel, respectively, of \( L\).
In what follows, without any possible confusion, all inner products are denoted by the same notation \( \langle \cdot, \cdot \rangle \) and all norms are denoted by the same notation \( \| \cdot \| \).

First, we present this proposition that demonstrates the relationship between a quaternionic super Hilbert space and the quaternionic Hilbert spaces that constitute it.
\begin{proposition}
The map
$$\begin{array}{rcl}
P_1 : \mathcal{H}_1 \oplus \mathcal{H}_2 &\longrightarrow &\mathcal{H}_1 \oplus \mathcal{H}_2\\
u \oplus v &\mapsto& u \oplus 0,
\end{array}$$
and the map
$$\begin{array}{rcl}
P_2 : \mathcal{H}_1 \oplus \mathcal{H}_2 &\longrightarrow& \mathcal{H}_1 \oplus \mathcal{H}_2\\
u \oplus v &\mapsto& 0 \oplus v,
\end{array}$$
are two orthogonal projections on \( \mathcal{H}_1 \oplus \mathcal{H}_2 \) with \( R(P_1) = \mathcal{H}_1 \oplus 0 \) and \( R(P_2) = 0 \oplus \mathcal{H}_2 \).
\end{proposition}
\begin{proof}
We have, clearly, \( P_1^2 = P_1 \) and \( P_2^2 = P_2 \). On the other hand, for \( u \oplus v, a \oplus b \in \mathcal{H}_1 \oplus \mathcal{H}_2 \), we have:
\[
\langle P_1(u \oplus v), a \oplus b \rangle = \langle u \oplus 0, a \oplus b \rangle
= \langle u, a \rangle + \langle 0, b \rangle
= \langle u, a \rangle + \langle v, 0 \rangle
= \langle u \oplus v, a \oplus 0 \rangle.
\]
Hence, for all \( a \oplus b \in \mathcal{H}_1 \oplus \mathcal{H}_2 \), \( P_1^*(a \oplus b) = a \oplus 0 = P(a \oplus b) \). Then \( P_1^* = P_1 \).

Similarly, we show that \( P_2^* = P_2 \). Thus, \( P_1 \) and \( P_2 \) are orthogonal projections.

It is clear that \( R(P_1) = \{ u \oplus 0 \,|\, u \in \mathcal{H}_1 \} := \mathcal{H}_1 \oplus 0 \) and \( R(P_2) = \{ 0 \oplus v \,|\, v \in \mathcal{H}_2 \} := 0 \oplus \mathcal{H}_2 \).
\end{proof}
The following proposition states that a mapping \( F = F_1 \oplus F_2 : \Omega \to \mathcal{H}_1 \oplus \mathcal{H}_2 \) is a Bessel mapping if and only if each component mapping \( F_1 \) and \( F_2 \) is a Bessel mapping for the associated space.
\begin{proposition}\label{prop6}
Let $F_1:\Omega\rightarrow \mathcal{H}_1$ and $F_2:\Omega\rightarrow \mathcal{H}_2$ be two mappings and define $F:=F_1\oplus F_2:\Omega\rightarrow \mathcal{H}_1\oplus \mathcal{H}_2$ for each $\omega\in \Omega$ by $(F_1\oplus f_2)(\omega)=F_1(\omega)\oplus F_2(\omega)$. The following statements are equivalent:
\begin{enumerate}
    \item $F_1\oplus F_2$ is a Bessel mapping for $\mathcal{H}_1 \oplus \mathcal{H}_2$.
    \item $F_1$ and $F_2$ are two Bessel mappings for $\mathcal{H}_1$ and $\mathcal{H}_2$ respectively.
\end{enumerate}
\end{proposition}
\begin{proof}

Assume that \(F\) is a Bessel mapping for  \(\mathcal{H}_1 \oplus \mathcal{H}_2\) with Bessel bound \(B\). Then $P_1F=F_1\oplus 0$ is a Bessel mapping for \(\mathcal{H}_1 \oplus 0\) and \(B\) is a Bessel bound. That means that for all \(u \oplus 0 \in \mathcal{H}_1 \oplus 0\),

\[
\int_{\Omega} \left|\langle  F_1(\omega) \oplus 0,u \oplus 0 \rangle\right|^2 d\mu(\omega)\leq B \|u \oplus 0\|^2.
\]

Hence, for all \(u \in \mathcal{H}_1\),

\[
\int_{\Omega} |\langle F_1(\omega),u \rangle|^2 d\mu(\omega) \leq B \|u\|^2.
\]

Then $F_1:\Omega\rightarrow \mathcal{H}_1$ is a Bessel mapping. By taking \(P_2\) instead of \(P_1\), we prove, similarly, that $F_2:\Omega\rightarrow \mathcal{H}_2$ is a Bessel mapping for \(\mathcal{H}_2\).

Conversely, assume that $F_1:\Omega\rightarrow \mathcal{H}_1$ and $F_2:\Omega\rightarrow \mathcal{H}_2$ are two Bessel mappings with Bessel bounds \(B_1\) and \(B_2\) respectively. Note \(B = 2\max\{B_1, B_2\}\). Then for all \(u \oplus v \in \mathcal{H}_1 \oplus \mathcal{H}_2\), we have:

$$\begin{array}{rcl}
\displaystyle{\int_{\Omega} \left\vert \langle F_1(\omega) \oplus F_2(\omega),u \oplus v \rangle\right\vert^2 d\mu(\omega)} &=& \displaystyle{\int_{\Omega} \left\vert\langle F_1(\omega),u \rangle + \langle F_2(\omega),v \rangle\right\vert^2 d\mu(\omega)}\\
&\leq& \displaystyle{\int_{\Omega} 2\left(|\langle  F_1(\omega),u \rangle|^2 + |\langle F_2(\omega),v \rangle|^2\right) d\mu(\omega)}\\
&\leq& 2B_1 \|u\|^2 + 2B_2 \|v\|^2\\
&\leq& B(\|u\|^2 + \|v\|^2) = B\|u \oplus v\|^2.
\end{array}$$
Hence, $F_1\oplus F_2:\Omega\rightarrow \mathcal{H}_1\oplus \mathcal{H}_2$ is a Bessel mapping.
\end{proof}
The following proposition expresses the operators associated with a super Bessel mapping in terms of the operators associated with its components.
\begin{proposition}\label{prop7}
Let \(F_1:\Omega\rightarrow \mathcal{H}_1\) and \(F_2:\Omega\rightarrow \mathcal{H}_2\) be two Bessel mappings. Let \(T_1\), \(T_2\), and \(T\) be the pre-frame  operators of \(F_1\), \(F_2\), and \(F_1\oplus F_2\) respectively. Let \(\theta_1\), \(\theta_2\), and \(\theta\) be the frame transforms of $F_1$, $F_2$ and $F_1\oplus F_2$, respectively. Let \(S_1\), \(S_2\), and \(S\) be the frame operators of $F_1$, $F_2$ and $F_1\oplus F_2$, respectively. Then:

\begin{enumerate}
    \item For all \(g\in L^2(\Omega,\mathbb{H})\), \(Tg = T_1g \oplus T_2g\).
    \item For all \(u \oplus v \in \mathcal{H}_1 \oplus \mathcal{H}_2\), \(\theta(u \oplus v) = \theta_1u + \theta_2v\).
    \item For all \(u \oplus v \in \mathcal{H}_1 \oplus \mathcal{H}_2\), \(S(u \oplus v) = S_1u + T_1\theta_2v \oplus S_2v + T_2\theta_1u\).
\end{enumerate}
\end{proposition}
\begin{proof}
\begin{enumerate}
\item Let $g\in L^2(\Omega,\mathbb{H})$, we have:
\[
Tg = \int_{\Omega}  (F_1(\omega) \oplus F_2(\omega)) g(\omega) d\mu(\omega) = \int_{\Omega} F_1(\omega) g(\omega) \oplus \int_{\Omega}  F_2(\omega)g(\omega) = T_1g\oplus T_2g.
\]

\item  Let \( u \oplus v \in \mathcal{H}_1 \oplus \mathcal{H}_2 \), we have:
\[
\theta(u \oplus v)(\omega) =  \langle  F_1(\omega) \oplus F_2(\omega),u \oplus v \rangle = \langle  F_1(\omega),u \rangle + \langle  F_2(\omega),v \rangle = \theta_1u(\omega) + \theta_2v(\omega).
\]

\item  Let \( u \oplus v \in \mathcal{H}_1 \oplus \mathcal{H}_2 \), we have:
$$\begin{array}{rcl}
S(u \oplus v) = T\theta(u \oplus v) = T(\theta_1(u) + \theta_2(v)) &=& T(\theta_1(u)) + T(\theta_2(v))\\
& =& T_1(\theta_1(u)) \oplus T_2(\theta_1(u)) + T_1(\theta_2(v)) \oplus T_2(\theta_2(v))\\
&=& S_1(u) \oplus T_2\theta_1(u) + T_1\theta_2(v) \oplus S_2(v)\\
&=&S_1(u) + T_1\theta_2(v) \oplus S_2(v) + T_2\theta_1(u).
\end{array}$$
\end{enumerate}
\end{proof}
The following theorem presents a necessary condition for a mapping $F:\Omega\rightarrow \mathcal{H}_1\oplus \mathcal{H}_2$  to be a continuous $K$-frame, where $K\in \mathbb{B}(\mathcal{H}_1\oplus \mathcal{H}_2)$.
\begin{theorem}
Let $K \in \mathbb{B}(\mathcal{H}_1 \oplus \mathcal{H}_2)$,  $F_1:\Omega\rightarrow \mathcal{H}_1$ and $F_2:\Omega\rightarrow \mathcal{H}_2$ be two mappings. If $F_1\oplus F_2$ is a  continuous $K$-frame for \( \mathcal{H}_1 \oplus \mathcal{H}_2 \) with  $K$-frame bounds  $A$ and $B$, then:

i. For all \( u \in \mathcal{H}_1 \),
\[
A \|K^*_1(u)\|^2 \leq \int_{\Omega} |\langle F(\omega), u \rangle|^2 d\mu(\omega)\leq B \|u\|^2.
\]

ii. For all \( v \in \mathcal{H}_2 \),
\[
A \|K^*_2(v)\|^2 \leq \int_{\Omega} |\langle F_2(\omega), v \rangle|^2 d\mu(\omega) \leq B \|v\|^2.
\]

Where \( K_1 : \mathcal{H}_1 \oplus \mathcal{H}_2 \rightarrow \mathcal{H}_1 \) and \( K_2 : \mathcal{H}_1 \oplus \mathcal{H}_2 \rightarrow \mathcal{H}_2 \) are the right $\mathbb{H}$-linear  bounded operators such that for all \( u \oplus v \in \mathcal{H}_1 \oplus \mathcal{H}_2 \),
\[
K(u \oplus v) = K_1(u \oplus v) \oplus K_2(u \oplus v).
\]
\end{theorem}
\begin{proof}
We have for all \( u \oplus v \in \mathcal{H}_1 \oplus \mathcal{H}_2 \),
\[
A \| K^*(u \oplus v) \|^2 \leq 
\int_{\Omega} 
| \langle F_1(\omega) \oplus F_2(\omega), u \oplus v \rangle |^2 d\mu(\omega) \leq B \| u \oplus v \|^2.
\]
i. In particular, for \( v = 0 \), we have for all \( u \in \mathcal{H}_1 \),
\[
A \| K^*(u \oplus 0) \|^2 \leq 
\int_{\Omega} 
| \langle F_1(\omega), u \rangle |^2 \leq B \| u \|^2.
\]
Let \( a \oplus b \in \mathcal{H}_1 \oplus \mathcal{H}_2 \) and let \( u \in \mathcal{H}_1 \), we have:
\[
\langle K(a \oplus b), u \oplus 0 \rangle = \langle K_1(a \oplus b) \oplus K_2(a \oplus b), u \oplus 0 \rangle 
= \langle K_1(a \oplus b), u \rangle 
= \langle a \oplus b, K_1^*(u) \rangle.
\]
Hence \( K^*(u \oplus 0) = K_1^*(u) \). Then for all \( u \in \mathcal{H}_1 \),
\[
A \| K_1^*(u) \|^2 \leq 
\int_{\Omega}
| \langle F_1(\omega), u \rangle |^2 d\mu(\omega) \leq B \| u \|^2.
\]
ii. In particular, for \( u = 0 \), we have for all \( v \in \mathcal{H}_2 \),
\[
A \| K^*(0 \oplus v) \|^2 \leq 
\int_{\Omega}
| \langle F_2(\omega), v \rangle |^2 \leq B \| v \|^2 d\mu(\omega).
\]
Let \( a \oplus b \in \mathcal{H}_1 \oplus \mathcal{H}_2 \) and let \( v \in \mathcal{H}_2 \), we have:
\[
\langle K(a \oplus b), 0 \oplus v \rangle = \langle K_1(a \oplus b) \oplus K_2(a \oplus b), 0 \oplus v \rangle 
= \langle K_2(a \oplus b), v \rangle 
= \langle a \oplus b, K_2^*(v) \rangle.
\]
Hence \( K^*(0 \oplus v) = K_2^*(v) \). Then for all \( v \in \mathcal{H}_2 \),
\[
A \| K_2^*(v) \|^2 \leq 
\int_{\Omega}
| \langle F_2(\omega), v \rangle |^2 d\mu(\omega) \leq B \| v \|^2.
\]
\end{proof}
The following result shows that a continuous $K_1\oplus K_2$-frame for $\mathcal{H}_1\oplus\mathcal{H}_2$ is necessarily the sum of a $K_1$-frame for $\mathcal{H}_1$ and a $K_2$-frame for $\mathcal{H}_2$.
\begin{corollary}\label{cor2}
Let \( K_1 \in \mathbb{B}(\mathcal{H}_1) \) and \( K_2 \in \mathbb{B}(\mathcal{H}_2) \). If $F:=F_1\oplus F_2:\Omega\rightarrow \mathcal{H}_1\oplus \mathcal{H}_2$ is a continuous $K_1\oplus K_2$-frame,  then $F_1:\Omega\rightarrow \mathcal{H}_1$ and $F_2:\Omega\rightarrow \mathcal{H}_2$ are  continuous $K_1$-frame and continuous $K_2$-frame, respectively.
\end{corollary}

\begin{lemma}\label{lem}
If $(\Omega,\mu)$ is a $\sigma$-finite measure space and $\mathcal{H}$ is a separable Hilbert space, then there exists a continuous frame for $\mathcal{H}$ with repect to $(\Omega,\mu)$.\\
\end{lemma}
\begin{proof}
Let $\{\Omega_n\}_{n\in \mathbb{N}}$ be a measurable partition of $\Omega$ such that for all $n\in \mathbb{N}$, $0< \mu(\Omega_n):=\mu_n \infty$ and $\{e_n\}_{n\in \mathbb{N}}$ be a Hilbert basis for $\mathcal{H}$. Define $F:\Omega\rightarrow \mathcal{H}$ as follows: $F(\omega)=e_n \displaystyle{\frac{1}{\sqrt{\mu_n}}}$ for all $\omega\in \Omega_n$. Then, for all $u\in \mathcal{H}$, 
$$
\begin{array}{rcl}
\displaystyle{\int_{\Omega}\vert \langle F(\omega),u\rangle \vert^2 d\mu(\omega)}&=&\displaystyle{\sum_{n\in\mathbb{N}}\int_{\Omega_n} \vert \langle F(\omega),u\rangle \vert^2 d\mu(\omega)}\\
&=&\displaystyle{\sum_{n\in\mathbb{N}}\mu_n .\frac{1}{\mu_n} \vert \langle e_n,u\rangle \vert^2}\\
&=&\| u\|^2.
\end{array}$$
Hence, $F:\Omega\rightarrow \mathcal{H}$ is a Parseval continuous frame.
\end{proof}

\begin{corollary}
Let $(\Omega,\mu)$ be a $\sigma$-finite measure space, $\mathcal{H}_1$ and $\mathcal{H}_2$ be two separable Hilbert spaces, \( K_1 \in \mathbb{B}(\mathcal{H}_1) \) and \( K_2 \in \mathbb{B}(\mathcal{H}_2) \). Then, there exist a continuous \( K_1 \)-frame $F_1:\Omega\rightarrow \mathcal{H}_1$ and a continuous $K_2$-frame $F_2:\Omega\rightarrow \mathcal{H}_2$ such that $F_1\oplus F_2:\Omega\rightarrow \mathcal{H}_1\oplus \mathcal{H}_2$ is a continuous  \( K_1 \oplus K_2 \)-frame for \( \mathcal{H}_1 \oplus \mathcal{H}_2 \).
\end{corollary}
\begin{proof}
Take any continuous frame  $F:\Omega \rightarrow \mathcal{H}_1 \oplus \mathcal{H}_2$ (exists by Lemma \ref{lem}). By Proposition \ref{prop2}, \((K_1\oplus K_2)F \) is a continuous \( K_1 \oplus K_2 \)-frame for \( \mathcal{H}_1 \oplus \mathcal{H}_2 \). That means that $K_1F\oplus K_2F$ is a continuous \( K_1 \oplus K_2 \)-frame for \( \mathcal{H}_1 \oplus \mathcal{H}_2 \). Then, by Corollary \ref{cor2}, $KF_1$ is a \( K_1 \)-frame for \( \mathcal{H}_1 \) and $KF_2$ is a \( K_2 \)-frame for \( \mathcal{H}_2 \).
\end{proof}
We will the following useful lemma.
\begin{lemma}
Let \( K_1 \in \mathbb{B}(\mathcal{H}_1) \) and \( K_2 \in \mathbb{B}(\mathcal{H}_2) \). Then \( (K_1 \oplus K_2)^* = K_1^* \oplus K_2^* \).
\end{lemma}
\begin{proof}
Let \( u \oplus v, a \oplus b \in \mathcal{H}_1 \oplus \mathcal{H}_2 \), we have:
\[
\begin{array}{rcl}
\langle K_1 \oplus K_2 (u \oplus v), a \oplus b \rangle & = & \langle K_1(u) \oplus K_2(v), a \oplus b \rangle \\
& = & \langle K_1(u), a \rangle + \langle K_2(v), b \rangle \\
& = & \langle u, K_1^*(a) \rangle + \langle v, K_2^*(b) \rangle \\
& = & \langle u \oplus v, K_1^*(a) \oplus K_2^*(b) \rangle \\
& = & \langle u \oplus v, (K_1^* \oplus K_2^*)(a \oplus b) \rangle.
\end{array}
\]
\end{proof}
The following proposition shows that there is no continuous \( K_1 \oplus K_2 \)-frame  in the form $F\oplus F:\Omega\rightarrow \mathcal{H}\oplus \mathcal{H}$ whenever \( K_1, K_2 \neq 0 \).
\begin{proposition}\label{prop8}
Let \( K_1, K_2 \in \mathbb{B}(\mathcal{H}) \) and $F:\Omega\rightarrow \mathcal{H}$ be a Bessel mapping.
Then:
\[
F\oplus F:\Omega\rightarrow \mathcal{H}\oplus \mathcal{H} \text{ is a continuous } K_1 \oplus K_2\text{-frame}  \iff K_1 = K_2 = 0.
\]
\end{proposition}
\begin{proof}
Assume that \( K_1 \neq 0 \) or \( K_2 \neq 0 \). We can suppose the case of \( K_1 \neq 0 \) and the other case will be obtained similarly. For \( u \in \mathcal{H} \) such that \( K_1^*(u) \neq 0 \), we have 
\[
\| K_1^* \oplus K_2^*(u \oplus (-u)) \|^2 = \| K_1^*(u) \|^2 + \| K_2^*(u) \|^2 \geq \| K_1^*(u) \|^2 \neq 0 
\]
and 
\[
\int_{\Omega} | \langle  F(\omega) \oplus F(\omega),u \oplus (-u) \rangle |^2 d\mu(\omega) = 0.
\]
Hence, $F\oplus F$ is not a continuous \( K_1 \oplus K_2 \)-frame.  Conversely, assume that \( K_1 = K_2 = 0 \). The result follows immediately from the fact that $F\oplus F$ is a Bessel mapping.
\end{proof}
Proposition \ref{prop8} showed that the direct sum of a continuous $K_1$-frame and a continuous $K_2$-frame is not necessary a continuous $K_1\oplus K_2$-frame for the right quaternionic super Hilbert space. The following result is a sufficient condition for a mapping $F_1\oplus F_2:\Omega\rightarrow \mathcal{H}_1\oplus \mathcal{H}_2$  to be a continuous \( K_1 \oplus K_2 \)-frame.
\begin{theorem}\label{thm8}
Let \( K_1 \in \mathbb{B}(\mathcal{H}_1) \) and \( K_2 \in \mathbb{B}(\mathcal{H}_2) \). Let $F_1:\Omega\rightarrow \mathcal{H}$ be a  continuous \( K_1 \)-frame for \( \mathcal{H}_1 \) and $F_2:\Omega\rightarrow \mathcal{H}_2$ be a continuous \( K_2 \)-frame for \( \mathcal{H}_2 \). Let \( \theta_1 \) and \( \theta_2 \) be the  transform operators of $F_1$ and $F_2$, respectively. If \( R(\theta_1) \perp R(\theta_2) \) in \( L^2(\Omega,\mathbb{H}) \), then $F_1\oplus F_2$ is a continuous \( K_1 \oplus K_2 \)-frame. 
\end{theorem}
\begin{proof}
Since $F_1:\Omega\rightarrow \mathcal{H}_1$ and $F_2:\Omega\rightarrow \mathcal{H}_2$ are Bessel mappings, then, by Proposition \ref{prop6}, $F_1\oplus F_2:\Omega\rightarrow \mathcal{H}_1\oplus \mathcal{H}_2$ is a Bessel mapping.  Denote by \( A_1 \) and \( A_2 \) the lower bounds for $F_1$ and $F_2$, respectively, and let \( A = \min \{ A_1, A_2 \} \). We have for all \( u \in \mathcal{H}_1, v \in \mathcal{H}_2 \), \( \langle \theta_1(u), \theta_2(v) \rangle = 0 \). Then for all \( u \in \mathcal{H}_1, v \in \mathcal{H}_2 \), \( \langle T_2 \theta_1(u), v \rangle = 0 \) and \( \langle T_1 \theta_2(v), u \rangle = 0 \), where \( T_1 \) and \( T_2 \) are the pre-frame  operators for $F_1$ and $F_2$, respectively. Hence \( T_2 \theta_1 =0_{\mathbb{B}(\mathcal{H}_1)}$ and $ T_1 \theta_2 = 0_{\mathbb{B}(\mathcal{H}_2)} \). By Proposition \ref{prop7}, the frame operator \( S \) of \(F_1\oplus F_2 \) is defined by 
\[
S(u \oplus v) = S_1(u) \oplus S_2(v) = (S_1 \oplus S_2)(u \oplus v),
\]
where \( S_1 \) and \( S_2 \) are the frame operators for $F_1$ and $F_2$, respectively. 

Then: 
$$\begin{array}{rcl}
\langle S(u \oplus v), u \oplus v \rangle &=& \langle S_1(u), u \rangle + \langle S_2(v), v \rangle\\ 
&\geq& A_1 \| K_1^*(u) \|^2 + A_2 \| K_2^*(v) \|^2\\ 
&\geq &A \| K_1^*(u) \oplus K_2^*(v) \|^2\\ 
&\geq& A \| (K_1 \oplus K_2)^*(u \oplus v) \|^2.
\end{array}$$

Hence, by theorem \ref{thm2}, $F_1\oplus F_2:\Omega\rightarrow \mathcal{H}_1\oplus \mathcal{H}_2$ is a continuous \( K_1 \oplus K_2 \)-frame.
\end{proof}
The following theorem presents a necessary condition for a mapping $F:\Omega\rightarrow \mathcal{H}_1\oplus \mathcal{H}_2$ to be a continuous $K_1\oplus K_2$-frame.
\begin{theorem}
Let \( K_1 \in \mathbb{B}(\mathcal{H}_1) \) and \( K_2 \in \mathbb{B}(\mathcal{H}_2) \). 
Let $F_1:\Omega\rightarrow \mathcal{H}_1$ be a continuous \( K_1 \)-frame 
and $F_2:\Omega\rightarrow \mathcal{H}_2$  be a continuous  \( K_2 \)-frame.  
If $F_1\oplus F_2$ is a continuous \( K_1 \oplus K_2 \)-frame, then:
\[
\begin{cases}
R(K_1) \subset T_2(N(T_1)),\\
R(K_2) \subset T_1(N(T_2)).
\end{cases}
\]
Where \( T_1 \) and \( T_2 \) are the pre-frame  operators for $F_1:\Omega\rightarrow \mathcal{H}_1$ and $F_2:\Omega\rightarrow \mathcal{H}_2$, respectively.
\end{theorem}
\begin{proof}
Assume that \( F_1 \oplus F_2 \) is a continuous \( K_1 \oplus K_2 \)-frame and let \( T \) be its synthesis operator. By Theorem \ref{thm2}, we know that \( R(K_1 \oplus K_2) \subset R(T) \). Then, for all \( u \oplus v \in \mathcal{H}_1 \oplus \mathcal{H}_2 \), there exists \( g \in L^2(\Omega, \mathbb{H}) \) such that
\[
K_1 \oplus K_2 (u \oplus v) = Tg.
\]
This implies, by Proposition \ref{prop7}, that for all \( u \oplus v \in \mathcal{H}_1 \oplus \mathcal{H}_2 \), there exists \( g \in L^2(\Omega, \mathbb{H}) \) such that
\[
K_1 u \oplus K_2 v = T_1 g \oplus T_2 g.
\]
By setting \( v = 0 \), there exists \( g \in L^2(\Omega, \mathbb{H}) \) such that
\[
K_1 u = T_1 g \quad \text{and} \quad 0 = K_2 v = T_2 g.
\]
Thus, we have \( K_1 u = T_1 g \) and \( g \in N(T_2) \). Therefore, 
\[
R(K_1) \subset T_1(N(T_2)).
\]
Similarly, by setting \( u = 0 \), there exists \( g \in L^2(\Omega, \mathbb{H}) \) such that
\[
0 = K_1 u = T_1 g \quad \text{and} \quad K_2 v = T_2 g.
\]
Hence, \( K_2 v = T_2 g \) and \( g \in N(T_1) \). Therefore,
\[
R(K_2) \subset T_2(N(T_1)).
\]
\end{proof}
The following result demonstrates that the non-minimality of the two  Bessel mappings $F_1:\Omega\rightarrow \mathcal{H}_1$ and $F_2:\Omega\rightarrow \mathcal{H}_2$ is a necessary condition for their direct sum to be a continuous \( K_1 \oplus K_2 \)-frame, assuming \( K_1 \neq 0 \) and \( K_2 \neq 0 \).

\begin{corollary} 
Let \( K_1 \in B(\mathcal{H}_1) \) and \( K_2 \in B(\mathcal{H}_2) \), and suppose that $F_1\oplus F_2:\Omega\rightarrow \mathcal{H}_1\oplus \mathcal{H}_2$ forms a continuous \( K_1 \oplus K_2 \)-frame. Then, the following statements hold:
\begin{enumerate}
    \item If $F_1:\Omega\rightarrow \mathcal{H}_1$ is \( K_1 \)-minimal, then \( K_2 = 0 \).
    \item If $F_2:\Omega\rightarrow \mathcal{H}_2$ is \( K_2 \)-minimal, then \( K_1 = 0 \).
    \item If both $F_1:\Omega\rightarrow \mathcal{H}_1$ is \( K_1 \)-minimal and $F_2:\Omega\rightarrow \mathcal{H}_2$ is \( K_2 \)-minimal, then both \( K_1 = 0 \) and \( K_2 = 0 \).
\end{enumerate}

In particular, if \( K_1 \neq 0 \) and \( K_2 \neq 0 \), and at least one of the Bessel mappings $F_1$ or $F_2$ is minimal, then $F_1\oplus F_2:\Omega\rightarrow \mathcal{H}_1\oplus \mathcal{H}_2$ cannot be a continuous \( K_1 \oplus K_2 \)-frame. 
\end{corollary}
The following result provides a necessary and sufficient condition for a continuous $F:\Omega \rightarrow \mathcal{H}_1\oplus \mathcal{H}_2$ to be a continuous \( K \)-minimal frame.

\begin{proposition}\label{prop9}
Let \( K \in B(\mathcal{H}_1 \oplus \mathcal{H}_2) \), and suppose \( F_1\oplus F_2:\Omega\rightarrow \mathcal{H}_1\oplus \mathcal{H}_2 \) is a continuous \( K \)-frame. The following conditions are equivalent:
\begin{enumerate}
    \item \( F_1\oplus F_2:\Omega\rightarrow \mathcal{H}_1\oplus \mathcal{H}_2 \) is a continuous \( K \)-minimal frame. 
    \item \( N(T_1) \cap N(T_2) = \{0\} \),
\end{enumerate}
where \( T_1 \) and \( T_2 \) are the pre-frame operators corresponding to $F_1:\Omega\rightarrow \mathcal{H}_1$ and \( F_2:\Omega\rightarrow \mathcal{H}_2 \), respectively.
\end{proposition}

\begin{proof}
By Proposition \ref{prop7}, for any \( g \in L^2(\omega,\mathbb{H}) \), the synthesis operator \( T \) acting on \( g \) satisfies \( Tg = T_1g \oplus T_2g \), where \( T_1 \) and \( T_2 \) are the pre-frame operators of $F_1$ and $F_2$, respectively. Thus, we have the relation \( N(T) = N(T_1) \cap N(T_2) \). Therefore, the equivalence follows.
\end{proof}
The following result gives a sufficient condition for the mapping $F_1\oplus F_2:\Omega\rightarrow \mathcal{H}_1\oplus \mathcal{H}_2$ to be a continuous \( K_1 \oplus K_2 \)-minimal frame. 
\begin{theorem}
Let \( K_1 \in B(\mathcal{H}_1) \) and \( K_2 \in B(\mathcal{H}_2) \). Suppose that \(F_1:\Omega\rightarrow \mathcal{H}_1 \) is a  continuous \( K_1 \)-frame  and \( F_2:\Omega\rightarrow \mathcal{H}_2 \) is a continuous \( K_2 \)-frame. Let \( \theta_1 \) and \( \theta_2 \) represent the  transform operators of \( F_1 \) and \( F_2 \), respectively.
If \( \overline{R(\theta_1)} = R(\theta_2)^\perp \), then the mapping \( F_1\oplus F_2:\Omega\rightarrow \mathcal{H}_1\oplus \mathcal{H}_2 \) is a continuous \( K_1 \oplus K_2 \)-minimal frame. 
\end{theorem}
\begin{proof}
Assume that \(\overline{ R(\theta_1)} = R(\theta_2)^\perp \), then  \( R(\theta_1) \perp R(\theta_2) \) and \( R(\theta_1)^\perp \cap R(\theta_2)^\perp = \{0\} \). Then, Theorem \ref{thm8} implies that \( F_1\oplus F_2:\Omega\rightarrow \mathcal{H}_1 \mathcal{H}_2\) is a \( K_1 \oplus K_2 \)-frame. Since \( N(\theta_1^*) = R(\theta_1)^\perp \) and \( N(\theta_2^*) = R(\theta_2)^\perp \), we have \( N(\theta_1^*) \cap N(\theta_2^*) = \{0\} \). This means that \( N(T_1) \cap N(T_2) = \{0\} \), where \( T_1 \) and \( T_2 \) are the pre-frame operators of $F_1$ and $F_2$, respectively. Hence, by Proposition \ref{prop9}, $F_1\oplus F_2:\Omega\rightarrow \mathcal{H}_1\oplus \mathcal{H}_2$ is a \( K_1 \oplus K_2 \)-minimal frame. 
\end{proof}
\section{Continuous $K_1$-duality, continuous $K_2$-duality and continuous  $K_1\oplus K_2$-duality}
In this section, $\mathcal{H}_1$ and $\mathcal{H}_2$ are two right quaternionic Hilbert spaces, $K_1\in \mathbb{B}(\mathcal{H}_1)$ and $K_2\in \mathbb{B}(\mathcal{H}_2)$. Given $F_1\oplus F_2:\Omega\rightarrow \mathcal{H}_1\oplus \mathcal{H}_2$, a continuous $K_1\oplus K_2$-frame. We will explore the relationship between the $K_1$-dual mapping to $F_1$, the $K_2$-dual mapping to $F_2$ and the $K_1\oplus K_2$-dual mapping to $F_1\oplus F_2$. For $G_1\in \mathbb{B}(\mathcal{H}_1,L^2(\Omega,\mathbb{H}))$ and $G_2\in \mathbb{B}(\mathcal{H}_2,L^2(\Omega,\mathbb{H}))$, we define $G_1\oplus G_2\in \mathbb{B}(\mathcal{H}_1\oplus \mathcal{H}_2,L^2(\Omega,\mathbb{H}))$ as follows: For all $u\oplus v\in \mathcal{H}_1\oplus \mathcal{H}_2$ and $\omega\in \Omega$, $(G_1\oplus G_2)(u\oplus v)(\omega):=G_1u(\omega)+G_2v(\omega).$\\
\begin{theorem}
Let \( K_1 \in \mathbb{B}(\mathcal{H}_1) \) and \( K_2 \in \mathbb{B}(\mathcal{H}_2) \) and $F_1\oplus F_2:\Omega\rightarrow \mathcal{H}_1\oplus \mathcal{H}_2$  be a continuous \( K_1 \oplus K_2 \)-frame.
If $G$ is a \( K_1 \oplus K_2 \)-dual mapping to \( F_1\oplus F_2 \) of the form $G_1\oplus G_2$, where $G_1\in \mathbb{B}(\mathcal{H}_1,L^2(\Omega,\mathbb{H}))$ and $G_2\in \mathbb{B}(\mathcal{H}_2,L^2(\Omega,\mathbb{H}))$, then \(G_1 \) is a \( K_1 \)-dual mapping to \( F_1:\Omega\rightarrow \mathcal{H}_1 \) and \( G_2 \) is a \( K_2 \)-dual mapping to \(F_2:\Omega\rightarrow \mathcal{H}_2 \).
\end{theorem}
\begin{proof}
Assume that $G_1\oplus G_2$ is a $K_1\oplus K_2$-dual mapping to $F_1\oplus F_2$, then $K_1\oplus K_2=T(G_1\oplus G_2)$, where $T$ is the pre-frame operator of $F_1\oplus F_2$. Hence, $K_1\oplus K_2=T_1(G_1\oplus G_2)\oplus T_2(G_1\oplus G_2)$. For all $u,v\in \mathcal{H}_1\oplus \mathcal{H}_2$, we have $K_1u\oplus K_2v=(T_1G_1u+T_1G_2v)\oplus (T_2G_1u+T_2G_2v)$. Thus, $K_1u=T_1G_1u+T_1G_2v$ and $K_2v=T_2G_1u+T_2G_2v$. Setting $v=0$ in the espression of $K_1$ and $u=0$ in the expression of $K_2$, we obtain for all $u\in \mathcal{H}_1$, $K_1u=T_1G_1u$ and for all $v\in \mathcal{H}_2$, $K_2v=T_2G_2v$. Hence, \(G_1 \) is a \( K_1 \)-dual mapping to \( F_1:\Omega\rightarrow \mathcal{H}_1 \) and \( G_2 \) is a \( K_2 \)-dual mapping to \(F_2:\Omega\rightarrow \mathcal{H}_2 \).
\end{proof}
One could question whether the converse of the previously stated theorem is true. The following result shows that this is not always the case, and it provides a necessary and sufficient condition under which the converse does hold.
\begin{theorem}
Let $F_1:\Omega\rightarrow \mathcal{H}_1$ and $\Omega\rightarrow \mathcal{H}_2$ be two continuous frames and $G_1$ is a $K_1$-dual mapping to $F_1$ and $G_2$ is a $K_2$-dual mapping to $F_2$. Then, the following statements are equivalent.
\begin{enumerate}
\item $F_1\oplus F_2:\Omega\rightarrow \mathcal{H}_1\oplus \mathcal{H}_2$ is a continuous frame and  $G_1\oplus G_2$ is $K_1\oplus K_2$-dual mapping.
\item $T_1G_2=0_{\mathbb{B}(\mathcal{H}_2)}$ and $T_2G_1=0_{\mathbb{B}(\mathcal{H}_1)}$.
\end{enumerate}
\end{theorem}
\begin{proof}
Since $G_1$ is a $K_1$-dual to $F_1$ and $G_2$ is a $K_2$-dual to $F_2$, then $K_1=T_1G_1$ and $K_2=T_2G_2$. Denote by $T$ the pre-frame operator of $F_1\oplus F_2$. Let $u\oplus v\in \mathcal{H}_1\oplus \mathcal{H}_2$. Then, $(K_1\oplus K_2)(u\oplus v)=T(G_1\oplus G_2)(u\oplus v)$, if and only if, $K_1u\oplus K_2v=(T_1G_1u+T_1G_2v)\oplus (T_2G_1u+T_2G_2v)$, if and only if, For all $u\in \mathcal{H}_1$ and $v\in \mathcal{H}_2$, $K_1u=T_1G_1u+T_1G_2v$ and $K_2v=T_2G_1u+T_2G_2v$, if and only if, $T_1G_2=0_{\mathbb{B}(\mathcal{H}_1)}$ and $T_2G_1=0_{\mathbb{B}(\mathcal{H}_1)}$. Hence, we obtain the requested equivalence.
\end{proof}

\medskip

	\section*{Acknowledgments}
	It is my great pleasure to thank the referee for his careful reading of the paper and for several helpful suggestions.
	
	\section*{Ethics declarations}
	
	\subsection*{Availablity of data and materials}
	Not applicable.
	\subsection*{Conflict of interest}
	The author declares that he has no competing interests.
	\subsection*{Fundings}
	Not applicable.

\end{document}